\documentclass[12pt]{amsart}
\usepackage{amssymb,amsmath}
\usepackage{verbatim}
\usepackage[mathscr]{eucal}
\usepackage{pgf}
\def\mL{L\kern-0.08cm\char39}

\newtheorem{thm}{Theorem}[section]

\newtheorem{cor}[thm]{Corollary}
\theoremstyle{definition}
	
\theoremstyle{remark}
\newtheorem{remark}[thm]{\bf Remark}
\begin{document}
\title{More on products of Baire spaces}
\author{Rui Li and L\'aszl\'o Zsilinszky}
%\centerline{\today}
\address{School of Statistics and Mathematics, Shanghai Finance University, Shanghai 201209, China}
\email{lir@sfu.edu.cn}

\address{Department of Mathematics and Computer Science, The University of
North Carolina at Pembroke, Pembroke, NC 28372, USA}
\email{laszlo@uncp.edu}



\begin{abstract} New results on the Baire product problem are presented. It is shown that an arbitrary product of almost locally ccc Baire spaces is Baire; moreover, the product of a Baire space and  a 1st countable space which is $\beta$-unfavorable in the strong Choquet game is Baire.
\end{abstract}

\maketitle

\section{introduction}
A topological space is a  {\it Baire space}
provided countable intersections of dense open subsets are dense \cite{HMC}. If the product $X\times Y$ is Baire, then $X,Y$ must  be Baire; however, the converse is not true in general. Indeed,
 {\it Oxtoby}  \cite{Ox} constructed, under CH, a Baire space with a non-Baire square, and
various absolute examples followed (see  \cite{Co}, \cite{FK}, \cite{Pol},  \cite{PVM}).
As a result, there has been a considerable effort to find various completeness properties for
the coordinate spaces  to get Baireness of the product (cf. \cite{KU}, \cite{Fr}, \cite{Ox},  \cite{HMC}, \cite{AL}, \cite{Wh2}, \cite{FK}, \cite{Pol}, \cite{FNR}, \cite{Zs0}, \cite{ChP}, \cite{Mo}, \cite{LM}). There have been two successful approaches in solving the product problem: given Baire spaces $X,Y$, either one adds some condition to $Y$ (such as 2nd countability \cite{Ox}, the uK-U property \cite{FNR}, having a countable-in-itself $\pi$-base \cite{Zs0}), or strengthens completeness of $Y$ (to \v Cech-completeness, (strong) $\alpha$-favorability \cite{AL},\cite{Wh2}, or more recently, to hereditary Baireness \cite{ChP},\cite{Mo},\cite{LM}).
It is the purpose of this paper to generalize these product theorems, as well as, show how a new  fairly weak completeness property of {\it $\beta$-unfavorability in the strong Choquet game } \cite{Po},\cite{De},\cite{Te} can be added to the list of spaces giving a Baire product.

Since Baire spaces can be characterized via the Banach-Mazur game, it is not surprising that topological games have been applied to attack the Baire product problem. Our results continue in this line of research (precise definitions will be given in the next section); in the games two players take countably many  turns in choosing objects from a topological space $X$: in the {\it strong Choquet game} \cite{Ch, Ke} player $\beta$ starts, and always chooses an open set $V$ and a point $x\in V$, then player $\alpha$ responds by choosing  an open set $U$ such that $x\in U\subseteq V$, next $\beta$ chooses an open set $V'$ and a point $x'$ with $x'\in V'\subseteq U$, etc.  Player $\alpha$ wins if the intersection of the chosen open sets is nonempty, otherwise, $\beta$ wins.

The strong Choquet game  provides a useful unifying platform for studying completeness-type properties, as the following two celebrated theorems demonstrate in a metrizable space $X$:
\begin{itemize}
\item $Ch(X)$ is $\alpha$-favorable if and only if $X$ is completely metrizable \cite{Ch},
\item $Ch(X)$ is $\beta$-unfavorable  if and only if $X$ is hereditarily Baire (i.e. the nonempty {\it closed} subspaces of $X$ are Baire) \cite{De,Te,Po}.
\end{itemize}

The {\it Banach-Mazur game} $BM(X)$ \cite{HMC} (also called the {\it Choquet game} \cite{Ke})  is played as $Ch(X)$, except that both $\beta,\alpha$ choose open sets only.  In a topological space $X$, $BM(X)$ is $\beta$-unfavorable iff  $X$ is a {\it Baire space} \cite{Ox2,Kr2,SR}; consequently, if $BM(X)$ is $\alpha$-favorable, then $X$ is a Baire space.

To put our results in perspective, recall that $X\times Y$ is a Baire space if $X$ is a Baire topological space and
\begin{itemize}
\item either $Y$ is a topological space such that $BM(Y)$ is $\alpha$-favorable (in particular, if $Ch(Y)$ is $\alpha$-favorable) \cite{Wh2},
\item or $Y$ is a hereditarily Baire space which is metrizable \cite{Mo}, or more generally, 1st countable $T_3$ space \cite{LM}.
\end{itemize}

Since there are spaces which are $\alpha$-favorable in the strong Choquet game but are not hereditarily Baire (the Michael line \cite{En}), as well as metric hereditarily Baire spaces, which are not $\alpha$-favorable in the Banach-Mazur game (a Bernstein set \cite{En}), being $\beta$-unfavorable in the strong Choquet game is distinct from both hereditary Baireness as well as being $\alpha$-favorable in the strong Choquet game, thus, it is natural to ask the status of this property in the Baire product problem. Our main result in Section 3 (Theorem \ref{prod}) implies the following:

\begin{thm}\label{ch}
Let $X$ be a Baire space, $Y$ be a 1st countable topological space such that  $Ch(Y)$ is $\beta$-unfavorable. Then $X\times  Y$ is a Baire space.
\end{thm}

%\begin{cor}\label{cor}Let $X$ be a Baire space and $\mathcal Y$ be a finite collection of spaces with a dense set of $W$-points which are $\beta$-unfavorable in the strong Choquet game. Then $X\times  \prod_{Y\in\mathcal Y} Y$ is a Baire space.\end{cor}

The proof works for finite products, but it does  not naturally extend to  infinite products, so we will separately consider the infinite product case in Section 4, using the idea of a {\it Krom space} (\cite{Kr1},\cite{GFG}), to obtain:

\begin{thm}\label{iprod}
Let $I$ be an index set, and  $X_i$ be an almost locally ccc Baire space (defined in Section 3) for each $i\in I$ . Then $\prod_i X_i$ is a Baire space.
\end{thm}

\section{Preliminaries}

Unless otherwise noted, all spaces are topological spaces. As usual, $\omega$ denotes the non-negative integers, and $n\ge 1$ will be considered as sets of predecessors $n=\{0,\dots,n-1\}$. Let $\mathcal B$ be a base for a topological space $X$, and denote \[\mathcal E=\mathcal E(X)=\mathcal E(X,\mathcal B)=\{(x,U)\in X\times \mathcal B: x\in U\}.\]
In the {\it strong Choquet game} $Ch(X)$ players $\beta$ and $\alpha$ alternate in choosing
$(x_n,V_n)\in \mathcal E$ and $U_n\in\mathcal B$, respectively, with $\beta$ choosing first, so that
for each $n<\omega$, $x_n\in U_n\subseteq V_n$, and $V_{n+1}\subseteq U_n$. The play
\[(x_0,V_0),U_0,\dots,(x_n,V_n),U_n,\dots\]
is won by $\alpha$, if $\bigcap_{n} U_n (=\bigcap_{n} V_n)\neq\emptyset$; otherwise, $\beta$ wins.

A {\it strategy} in $Ch(X)$ for $\alpha$ (resp. $\beta$) is a function $\sigma: \mathcal E^{<\omega}\to\mathcal B$
(resp. $\sigma: \mathcal B^{<\omega}\to\mathcal E$) such that
$x_n\in \sigma((x_0,V_0),\dots,(x_n,V_n))\subseteq V_n$ for all
$((x_0,V_0),\dots,(x_n,V_n))\in \mathcal E^{<\omega}$ (resp. $\sigma(\emptyset)=(x_0,V_0)$ and
$V_n\subseteq U_{n-1}$, where
$\sigma(U_{0},\dots,U_{n-1})=(x_n,V_n)$, for all
$(U_{0},\dots,U_{n-1})\in \mathcal B^n$, $n\ge 1$).
A strategy $\sigma$ for $\alpha$ (resp. $\beta$) is a {\it winning strategy},
if $\alpha$ (resp. $\beta$) wins every play of $Ch(X)$
compatible with $\sigma$, i.e. such that  $\sigma((x_0,V_0),\dots,(x_n,V_n))=U_n$ for all $n<\omega$
(resp. $\sigma(\emptyset)=(x_0,V_0)$ and $\sigma(U_{0},\dots,U_{n-1})=(x_n,V_n)$ for all $n\ge 1$).
 We will say that {\it $Ch(X)$ is $\alpha$-, $\beta$-favorable}, respectively, provided $\alpha$, resp. $\beta$, has a winning strategy in $Ch(X)$.

The {\it Banach-Mazur game} $BM(X)$ \cite{HMC}  is played similarly to $Ch(X)$, the only difference is that both $\beta,\alpha$ choose open sets from a fixed $\pi$-base of $X$. Winning strategies, $\alpha$-, and $\beta$-favorability of $BM(X)$ can be defined analogously to $Ch(X)$.

In the {\it Gruenhage game} $G(X)$ \cite{Gr} given a point $x\in X$, at the $n$-th round   Player I picks an open neighborhood $U_n$ of $x$, and Player II chooses $x_n\in U_n$. Player I wins if the sequence $(x_n)_n$ converges to $x$, otherwise, Player II wins; $x$ is a {\it $W$-point} \cite{Sh}, provided Player I has a winning strategy $W_x:X^{<\omega}\to \{\text{open neighborhoods of}\ x\}$ in $G(X)$ at $x$.

Given a topological space $(X,\tau)$, consider the ultrametric space $\tau^{\omega}$, where $\tau$ has the discrete topology.  For every $n<\omega$ denote
\begin{gather*}
\downarrow\tau^n=\{f\in (\tau\setminus \{\emptyset\})^n: f(k+1)\subseteq f(k) \ \text{whenever} \ k\le n\}, \ \text{and}\\
 \downarrow\tau^{\omega}=\{f\in (\tau\setminus \{\emptyset\})^{\omega}: f(k+1)\subseteq f(k) \ \text{whenever} \ k<\omega\}.
\end{gather*}
The {\it Krom space} \cite{Kr1,GFG} of $X$  is defined as
\[
\mathcal K(X)=\{f\in \downarrow\tau^{\omega}: \bigcap_n f(n)\neq\emptyset
\}.
\]
Note that a base of neighborhoods at $f\in \mathcal K(X)$ is $\{[f\restriction_{n+1}]: n<\omega\}$, where
\[
[f\restriction_{n+1}]=\{g\in\mathcal K(X): g\restriction_{n+1}=f\restriction_{n+1}\}.
\]
Put differently, a base for $\mathcal K(X)$ is the family of all sets $[f]$, $f\in\bigcup_n\downarrow\tau^{n}$ where, if $n<\omega$ and  $f\in \downarrow\tau^n$, then
\[
[f]=\{g\in\mathcal K(X): g\restriction _{n+1}=f\}.
\]
Given a base $\mathcal B$ for $X$, we will also consider the following subspace of $\mathcal K(X)$:
\[
\mathcal K_{\mathcal B}^0(X)=\{f\in \mathcal K(X)\cap \mathcal B^{\omega}:  (f(n))_n \ \text{is a neighborhood base at each} \ x\in\bigcap_{n}f(n) \}.
\]
{\it Krom's Theorem} \cite[Theorem 3]{Kr1} states that for  topological spaces $X,Y$, $X\times Y$ is a Baire space iff $X\times \mathcal K (Y)$ is a Baire space iff $\mathcal K(X)\times \mathcal K(Y)$ is a Baire space.

Given a set $S$, $s\in S$, $n<\omega$, and $f\in S^n$, the notation $f^{\frown} s$ stands for the $g\in S^{n+1}$ for which $g\restriction _n=f\restriction _n$, and $g(n)=s$.
\vskip 1pc

\section{Finite Baire products}

\vskip 6pt
 We will say that a space is {\it almost locally ccc}, provided every open set contains an open ccc subspace i.e., if the space has a $\pi$-base of open ccc subspaces. This property is strictly weaker than being {\it almost locally uK-U} \cite{Zs0} (see \cite[Examples 1,2]{FNR}), as well as  {\it having a countable-in-itself $\pi$-base} \cite{Zs0} (i.e., having a $\pi$-base each member of which contains countably many members of said $\pi$-base -- this property is also termed {\it having a locally countable pseudo-base} in \cite{Ox}), which are known to produce Baire products (see \cite[Theorem 2]{Ox}, \cite[Property 1]{FNR} \cite[Proposition 4]{Zs0}). Since these properties all coincide in Baire metric spaces (see \cite[Proposition 3]{Zs0}), a simple observation about Krom spaces immediately yields a generalization of these Baire product theorems:

\begin{thm}\label{krom}
Let $X,Y$ be a Baire spaces, and $Y$  be almost locally ccc. Then $X\times Y$ is a Baire space.
\end{thm}
\begin{proof}  First note that $\mathcal K(Y)$ has a countable-in-itself $\pi$-base: indeed, let $f\in \downarrow \tau^{n}$  for some $n<\omega$,  choose $U\subset f(n)$ which is ccc, and define $f_0=f^\frown U$. Consider a pairwise disjoint open partition $\{[g]: g\in J\}$ of $[f_0]$, where $J\subset \bigcup_{n<\omega}\downarrow\tau^n$. For each $g\in J$ let $n_g<\omega$ be such that $g\in \downarrow\tau^{n_g} $; then $\{g(n_g): g\in J\}$ is a pairwise disjoint open partition of $U$, which must be countable, and so is $\{[g]: g\in J\}$; thus, $\mathcal K(Y)$ is an almost locally ccc metric space, and so it has a countable-in-itself $\pi$-base.
%Conversely, if $\mathcal K(Y)$ is almost locally ccc, and $U\subset Y$ is open, find $n<\omega$ and  $f\in \downarrow \tau^{n}$ so that $[f]$ is ccc and $[f]\subset [(U)]$. If $\mathcal V$ is a pairwise disjoint open partition of $f(n)$, then $\{[f^\frown V]: V\in\mathcal V\}$ is a pairwise disjoint open partition of $[f]$, thus, it is countable, and so is $\mathcal V$. This makes $f(n)$ a ccc open subset of $U$, so $Y$ is almost locally ccc.

It follows from Krom's theorem that $\mathcal K(Y)$ is a Baire space, moreover, by \cite[Theorem 2]{Ox}, $X\times \mathcal K(Y)$ is a Baire space, which it turn implies $X\times Y$ is a Baire space by Krom's theorem.
\end{proof}

An approach involving the strong Choquet game yields a different kind of generalization of Baire product theorems (cf. \cite{AL}, \cite{Wh2},\cite{Mo},\cite{LM},\cite{ChP}):

\begin{thm}\label{prod}Let $X$ be a Baire space and $Y$ have a dense set of $W$-points and $Ch(Y)$ be $\beta$-unfavorable. Then $X\times  Y$ is a Baire space.
\end{thm}

\begin{proof} Denote by $\tau_X,\tau_Y$ the nonempty open subsets of $X,Y$, respectively. Let $\{O_n:n<\omega\}$ be a decreasing sequence of dense open subsets of $X\times Y$, and pick $U\in\tau_X, V\in\tau_Y$.
If $y\in Y$ is a $W$-point, denote by $W_y$ a winning strategy for the open-set picker in the Gruenhage game at $y$.

Define a tree $T\subset \omega^{<\omega}$ as follows: let  $T_0=\{\emptyset\}$ be the root of the tree, and $T_1=\{(0)\}$ its first level; further, given level $T_n$ for some  $n\ge 1$, and $t=(t_0,\dots,t_k)\in T_n$, define the immediate successors of $t$ as $t^-=(t_0,\dots,t_k,0)$ and $t^+=(t_0,\dots,t_k+1)$, and put $T_{n+1}=\{t^-,t^+: t\in T_n\}$. For example, if $t=(0)\in T_1$, then $t^-=(0,0), \ t^+=(1)$; further, if $t=(0,0)\in T_2$, then $t^-=(0,0,0), \ t^+=(0,1)$, and if $t=(1)\in T_2$, then $t^-=(1,0), \ t^+=(2)$, etc.
It follows that each $t\in T_n$ for $n\ge 1$ has a "source`` $s_t\in T_k$ for some $0\le k<n$ such that $t=s_t^\frown (n-k-1)$; in other words, $s_t$ is  the node where the last minus-branching occurs before $t$; so, $s_{(0)}=\emptyset, s_{(0,0)}=(0), s_{(1)}=\emptyset, s_{(0,0,0)}=(0,0), s_{(0,1)}=(0), s_{(1,0)}=(1), s_{(2)}=\emptyset$, etc.

We will define a strategy $\sigma_X$ for $\beta$ in $BM(X)$:
first, pick a $W$-point $y_{\emptyset}\in V$ and denote $V_{\emptyset}=V$. Then choose
$U_0\in\tau_X$,  $V_{(0)}\in\tau_Y$  so that $U_0\times V_{(0)}\subseteq O_0\cap U\times V$, pick a $W$-point $y_{(0)}\in V_{(0)}$ and put $\sigma_X(\emptyset)=U_0$.
Given $A_0\in\tau_X, A_0\subseteq U_{0}$, find $U_{1}\in\tau_X$ and
$V_{t}\in\tau_Y$ for each $t\in T_2$ so that
\begin{gather*}
U_{1}\times V_{(0,0)}\subseteq O_1\cap [A_0\times V_{(0)}]\\
U_{1}\times V_{(1)}\subseteq O_1\cap [A_0\times V_{\emptyset}\cap  W_{y_{\emptyset}}(y_{(0)})];
\end{gather*}
moreover, pick a $W$-point $y_t\in V_t$ for each $t\in T_2$,
and put $\sigma_X(A_0)=U_1$.

Assume that for some $n\ge 1$,  and given $A_0,\dots,A_{n-1}\in\tau_X$, we have constructed $U_n\in\tau_X$ along with $(y_t,V_t)\in\mathcal E(Y)$  for each $t\in T_{n+1}$ so that each $y_t$ is a $W$-point,
 \[U_n=\sigma_X(A_0,\dots,A_{n-1}),\]
 and for each $t\in T_{n}$
\begin{gather}
U_n\times V_{t^-}\subseteq O_n\cap [A_{n-1}\cap V_t],\label{O}\\
U_n\times V_{t^+}\subseteq O_n\cap [A_{n-1}\times V_{s_t}\cap W_{y_{s_t}}(y_{s_t^\frown 0},\dots, y_t)]\label{W}.
\end{gather}

Let $A_n\in\tau_X, A_n\subseteq U_n$ be given, and denote $T_{n+1}=\{t_1,\dots,t_{2^n}\}$.
Using density of $O_{n+1}$ repeatedly, we can define a decreasing sequence $\{H_i\in \tau_X: i\le 2^{n+1}\}$, where $H_0=A_n$, as well as $V_{t_i^-},V_{t_i^+}\in\tau_Y$ so that for all $1\le i\le 2^{n}$
\begin{gather*}
H_i\times V_{t_i^-}\subseteq O_{n+1}\cap [H_{i-1}\times V_{t_i}],\\
H_{i+2^n}\times V_{t_i^+}\subseteq O_{n+1}\cap [H_{i+2^n-1}\times V_{s_{t_i}}\cap W_{y_{s_{t_i}}}(y_{s_{t_i}^\frown 0},\dots, y_{t_i})].
\end{gather*}
Then for $U_{n+1}=H_{2^{n+1}}$  and each $t\in T_{n+1}$ we have
\begin{gather*}
U_{n+1}\times V_{t^-}\subseteq O_{n+1}\cap [A_{n}\cap V_t]\\
U_{n+1}\times V_{t^+}\subseteq O_{n+1}\cap [A_{n}\times V_{s_t}\cap W_{y_{s_t}}(y_{s_t^\frown 0},\dots, y_t)].
\end{gather*}
Pick a $W$-point $y_t\in V_t$ for each $t\in T_{n+2}$, and
define $\sigma_X(A_0,\dots,A_n)=U_{n+1}$, which concludes the definition of  $\sigma_X$.
Notice that, by (\ref{W}),
\begin{equation}
(y_{t^{\frown}k})_k  \ \text{ converges to}\ y_t \ \text{for each} \ t\in T.\label{conv}
\end{equation}

Since $X$ is a Baire space, there is play $U_0,A_0,\dots,U_n,A_n,\dots$ of $BM(X)$ compatible with $\sigma_X$ that $\beta$ looses, i.e. there is some $x\in \bigcap_n U_n$.

We will define  a strategy $\sigma_Y$ for $\beta$ in $Ch(Y)$. First,
put $\sigma_Y(\emptyset)=(z_0,W_0)$, where $z_0=y_{\emptyset}$, and $W_0=V_{\emptyset}$. Let $B_0\in\tau_Y$ with $z_0\in B_0\subseteq W_0$ be given. Using (\ref{conv}), we can define
\[
k_{B_0}=\min\{k\ge 0: y_{(k)}\in B_0\}, \ z_1=y_{(k_{B_0})}, \ W_1=B_0\cap V_{(k_{B_0})},
\]
and put $\sigma_Y(B_0)=(z_1,W_1)$. Assume that $\sigma_Y(B_0,\dots,B_{n-1})=(z_n,W_n)\in\mathcal E(Y)$ have been defined for some $n\ge 1$ and $B_0,\dots,B_{n-1}\in\tau_Y$ so that
\[z_n=y_{(k_{B_0},\dots,k_{B_{n-1}})}, \ \text{and} \
W_n=B_{n-1}\cap V_{(k_{B_0},\dots,k_{B_{n-1}})}\]
for appropriate $k_{B_i}\ge i$ for each $i<n$. Let $B_{n}\in\tau_Y$ be such that $z_n\in B_{n}\subseteq W_n$, then for $t=(k_{B_0},\dots,k_{B_{n-1}})$,  $(y_{t^{\frown} k})_k$ converges to $z_n=y_t$ by (\ref{conv}), so we can define
\[k_{B_{n}}=\min\{k\ge n:  y_{t^{\frown} k}\in B_{n}\}, \ z_{n+1}=y_{t^{\frown} k_{B_{n}}}, \ W_{n+1}=B_{n}\cap V_{t^{\frown} k_{B_{n}}},\]
and put $\sigma_Y(B_0,\dots,B_{n})=(z_{n+1},W_{n+1})$.

Since $\sigma_Y$ cannot be a winning strategy for $\beta$ in $Ch(Y)$, there is a play
\[(z_0,W_0),B_0,(z_1,W_1),\dots,B_n,(z_{n+1},W_{n+1}),\dots\] of $Ch(Y)$
compatible with $\sigma_Y$ that $\beta$ looses. Then there is some $y\in \bigcap_n W_n\subseteq \bigcap_n
V_{(k_{B_0},\dots,k_{B_{n}})}$, so (\ref{O}) and  (\ref{W}) imply that $(x,y)\in U\times V\cap \bigcap_n O_n$, and we are done.
\end{proof}

The proof of Theorem \ref{ch} immediately follows, which in turn implies the following (recall, that a  space is $R_0$ \cite{Da}, when every open subset contains the closure of each of its points):

\begin{cor}\label{hbproduct} Let $X$ be a Baire space, and $Y$  a 1st countable hereditarily Baire $R_0$-space. Then $X\times Y$ is a Baire space.\end{cor}

\begin{proof} It suffices to note that a 1st countable hereditarily Baire $R_0$-space $Y$ is $\beta$-unfavorable in $Ch(Y)$ by \cite[Corollary 3.8.]{Zs1}; thus,  Theorem \ref{ch} applies.
\end{proof}

\begin{remark}
Some of the results in \cite{LM} are similar in flavor to the above results, in particular, \cite[Theorem 4.4]{LM} states, that {\it if $X$ is a Baire space, and $Y$ is a $T_3$-space possessing a rich family $\mathcal F$ of Baire spaces (i.e. $\mathcal F$ consists of nonempty separable closed Baire subspaces of $X$ such that if $Y\subset X$ is separable, then $Y\subseteq F$ for some $F\in\mathcal F$, moreover, $\overline{\bigcup_{n<\omega} F_n}\in\mathcal F$ whenever $\{F_n:n<\omega\}\subset \mathcal F$), then $X\times Y$ is a Baire space.} The next example shows that our results are different (although overlapping), since spaces that are $\beta$-unfavorable in the strong Choquet game are not directly connected to spaces having rich Baire families.
Indeed, there exists a $T_1$-space $X$ with  no rich Baire family which is  $\alpha$-favorable in $Ch(X)$:
to see this, let $\mathbb Q$ be the rationals and $D$ an uncountable set. Define  $X=\mathbb Q\cup D$, let elements of $D$ be isolated, and a neighborhood base at $q\in \mathbb Q$ be of the form $I\cup D\setminus C$, where $I\subseteq \mathbb Q$ is a Euclidean open neighborhood of $q$, and $C\subset D$ is countable. Then

$\bullet$ $X$ is strongly $\alpha$-favorable: define a tactic $\sigma$ for $\alpha$ in $Ch(X)$ as follows
\[
\sigma(x,V)=\begin{cases}
\{x\}, &\text{if} \ x\in D,\\
V, &\text{if} \ x\in \mathbb Q.
\end{cases}
\]
Each play of $Ch(X)$ compatible with $\sigma$ contains an element of $D$ in the intersection, so $\sigma$ is a winning tactic for $\alpha$.

$\bullet$ $X$ has no rich Baire family: indeed, for every separable  $S\supseteq \mathbb Q$ we have $S=\mathbb Q\cup C$ for some countable $C\subset D$. It follows that if $I\subseteq \mathbb Q$ is a Euclidean open neighborhood of some $q\in \mathbb Q$, then it is also an open set in $S$ (since $I=S\cap (I\cup D\setminus C)$), and  of the 1st category in $S$, thus, $S$ is not a Baire space.
\end{remark}

\begin{remark}
It is known that hereditary Baireness is not a stand-alone topological property that gives a Baire product since, under (CH), there is a hereditarily Baire space with a non-Baire square \cite{Wh1}; however, it is an open question weather $X\times Y$ is Baire if $X$ is Baire and $Ch(Y)$ is $\beta$-unfavorable.
\end{remark}

\vskip 1pc

\section{Infinite Baire products}

The following is the arbitrary product version of Krom's Theorem:

\begin{thm}\label{krominfi} Let $I$ be an index set. Then $\prod_{i\in I} X_i$ is a Baire space if and only if   $\prod_{i\in I} \mathcal K(X_i)$ is a Baire space.
\end{thm}

\begin{proof} Denote $X=\prod_{i\in I} X_i$ and $X^*=\prod_{i\in I} \mathcal K(X_i)$.

$\bullet$
Assume that $\beta$ has a winning strategy $\sigma$ in $BM(X)$, and define a strategy $\sigma^*$ for $\beta$ in $BM(X^*)$ as follows: if $\sigma(\emptyset)=\prod_{i\in I_0}V_{0,i}\times \prod_{i\notin I_0} X_i$ for some finite $I_0\subset I$ and  $V_{0,i}\in\mathcal B_i$, define
\[
\sigma^*(\emptyset)=\prod_{i\in I_0}V_{0,i}^*\times \prod_{i\notin I_0} \mathcal K(X_i), \ \text{where} \ V_{0,i}^*=[(V_{0,i})]
\]
If $U_0^*\subseteq \sigma^*(\emptyset)$ is $\alpha$'s response in $BM(X^*)$, then $U_0^* =\prod_{i\in J_0}U_{0,i}^*\times \prod_{i\notin J_0} \mathcal K(X_i)$ for some finite $J_0\supset I_0$, and for all $i\in J_0$, $U_{0,i}^*=[(U_{0,i}(0),\dots,U_{0,i}(m_{0,i}))]$ for some decreasing $X_i$-open $U_{0,i}(0),\dots,U_{0,i}(m_{0,i})$ and $m_{0,i}\ge 0$, where $U_{0,i}(0)=V_{0,i}$ for all $i\in I_0$.
Denote $U_0=\prod_{i\in J_0} U_{0,i}(m_{0,i})\times\prod_{i\notin J_0} X_i$ and let
\[
\sigma(U_0)=\prod_{i\in I_1} V_{1,i}\times \prod_{i\notin I_1} X_i,
\]
where $I_1\supseteq J_0$ is finite, $V_{1,i}\in\mathcal B_i$ for each $i\in I_1$ and $V_{1,i}\subseteq U_{0,i}(m_{0,i})$ whenever $i\in J_0$. Define
\begin{gather*}
\sigma^*(U_0^*)=\prod_{i\in I_1} V_{1,i}^*\times\prod_{i\notin I_1} \mathcal K(X_i),  \ \text{where}\\
V_{1,i}^*=\begin{cases}
[(U_{0,i}(0),\dots,U_{0,i}(m_{0,i}),V_{1,i})], \ &\text{if} \ i\in J_0,\\
[(V_{1,i})],  \ &\text{if} \ i\in I_1\setminus J_0.
\end{cases}
\end{gather*}

Proceeding inductively, we can define $\sigma^*$ so that whenever  $k<\omega$, and
\[
U_k^*=\prod_{i\in J_k}U_{k,i}^*\times \prod_{i\notin J_k} \mathcal K(X_i)
\]
is given for some finite  $J_k$, and for all $i\in J_k$, $U_{k,i}^*=[(U_{k,i}(0),\dots,U_{k,i}(m_{k,i}))]$ for decreasing $X_i$-open $U_{k,i}(0),\dots,U_{k,i}(m_{k,i})$ and $m_{k,i}\ge 0$,
then
\[
\sigma^*(U_0^*,\dots,U_k^*)=\prod_{i\in I_{k+1}} V_{k+1,i}^*\times\prod_{i\notin I_{k+1}} \mathcal K(X_i)
\]
 have been chosen, where $I_{k+1}\supseteq J_k$ is finite, and
\[
V_{k+1,i}^*=\begin{cases}
[(U_{k,i}(0),\dots,U_{k,i}(m_{k,i}),V_{k+1,i})], \ &\text{if} \ i\in J_k,\\
[(V_{k+1,i})],  \ &\text{if} \ i\in I_{k+1}\setminus J_k
\end{cases}\\
\]
is such that
\[
\prod_{i\in I_{k+1}} V_{k+1,i}\times \prod_{i\notin I_{k+1}} X_i=
\sigma\left (U_0,\dots U_k\right ).
\]
where $U_j=\prod_{i\in J_j} U_{j,i}(m_{j,i})\times\prod_{i\notin J_j} X_i$ for all $j\le k$.
We will show that $\sigma^*$ is a winning strategy for $\beta$ in $MB(X^*)$: indeed, take a play of $MB(X^*)$
\[
\sigma^*(\emptyset),U_0^*,\dots,U_n^*,\sigma^*(U_0^*,\dots,U_n^*),\dots
\]
compatible with $\sigma^*$, and assume there is some $f\in \bigcap_n \sigma^*(U_0^*,\dots,U_n^*)$. Then for each $i\in I$, $f(i)\in \mathcal K(X_i)$ so we can pick some $x_i\in\bigcap_n f(i)(n)$. Moreover, if $i\in I_k$ for a given $k<\omega$, then $x_i\in \bigcap_{n\ge k} V_{n,i}$, so $(x_i)_{i\in I}\in \prod_{i\in I_{k}} V_{k,i}\times \prod_{i\notin I_{k+1}} X_i$, thus, $(x_i)_{i\in I}\in \bigcap_n \sigma(U_0,\dots U_k)$ which is impossible, since $\sigma$ is a winning strategy for $\beta$ in $BM(X)$.

$\bullet$ Assume that $\beta$ has a winning strategy $\sigma^*$ in $BM(X^*)$, and define a strategy $\sigma$ for $\beta$ in $BM(X)$ as follows: if $\sigma^*(\emptyset)=\prod_{i\in I_0}V_{0,i}^*\times \prod_{i\notin I_0} \mathcal K(X_i)$, where for all $i\in I_0$, $V_{0,i}^*=[(V_{0,i}(0),\dots,V_{0,i}(m_{0,i}))]$, define $\sigma(\emptyset)=\prod_{i\in I_0}V_{0,i}(m_{0,i})\times \prod_{i\notin I_0} X_i$. Let $U_0=\prod_{i\in J_0} U_{0,i}\times  \prod_{i\notin J_0} X_i$ be $\alpha$'s response in $BM(X)$. Then $J_0\supseteq I_0$ is finite and $U_{0,i}\subseteq V_{0,i}(m_{0,i})$ for all $i\in I_0$. Define
\[
U_{0,i}^*=\begin{cases} [(V_{0,i}(0),\dots,V_{0,i}(m_{0,i}),U_{0,i})], &\text{for all} \ i\in I_0,\\
[(U_{0,i})], &\text{for all} \ i\in J_0\setminus I_0,
\end{cases}
\]
and let \[
\sigma^*\left (\prod_{i\in J_0} U_{0,i}^*\times \prod_{i\notin J_0} \mathcal K(X_i)\right )=\prod_{i\in I_1} V_{1,i}^*\times \prod_{i\notin I_1}  \mathcal K(X_i),\]
where $V_{1,i}^*=[(V_{1,i}(0),\dots,V_{1,i}(m_{1,i}))]$ whenever $i\in I_1$. Define
\[
\sigma(U_0)=\prod_{i\in I_1} V_{1,i}(m_{1,i})\times\prod_{i\notin I_1} X_i.
\]
Proceeding inductively, assume that whenever  $k\ge 1$, and $j<k$, then $\sigma(U_0,\dots,U_{j})=\prod_{i\in I_{j+1}} V_{j,i}(m_{j,i})\times\prod_{i\notin I_{j+1}} X_i$ is defined, and let
$U_k=\prod_{i\in J_k} U_{k,i}\times  \prod_{i\notin J_k} X_i$ be $\alpha$'s next step in $BM(X)$. Then $J_k\supseteq I_k$ is finite and $U_{k,i}\subseteq V_{k,i}(m_{k,i})$ for all $i\in I_k$. Define
\[
U_{k,i}^*=\begin{cases} [(V_{k,i}(0),\dots,V_{k,i}(m_{k,i}),U_{k,i})], &\text{for all} \ i\in I_k,\\
[(U_{k,i})], &\text{for all} \ i\in J_k\setminus I_k,
\end{cases}
\]
and let
\[
\sigma^*\left (\prod_{i\in J_k} U_{k,i}^*\times \prod_{i\notin J_k} \mathcal K(X_i)\right )=\prod_{i\in I_{k+1}} V_{k+1,i}^*\times \prod_{i\notin I_{k+1}}  \mathcal K(X_i),\]
where $V_{k+1,i}^*=[(V_{k+1,i}(0),\dots,V_{k+1,i}(m_{k+1,i}))]$ whenever $i\in I_{k+1}$. Define
\[
\sigma(U_0,\dots,U_k)=\prod_{i\in I_{k+1}} V_{k+1,i}(m_{k+1,i})\times\prod_{i\notin I_{k+1}} X_i.
\]
We will show that $\sigma$ is a winning strategy for $\beta$ in $BM(X)$: take a play
\[
\sigma(\emptyset),U_0,\dots,U_n,\sigma(U_0,\dots,U_n),\dots
\]
compatible with $\sigma$, and assume there is some $(x_i)_{i\in I}\in\bigcap_n \sigma(U_0,\dots,U_n)$.
For all $k<\omega$ and $i\in I_k$, define a decreasing sequence of $X_i$-open sets $f(i)$ so that $f(i)\restriction_{n}=(V_{n,i}(0),\dots,V_{n,i}(m_{n,i}))$ for all $n<\omega$, moreover, if $i\in I\setminus \bigcup_{k} I_k$, put $f(i)=(X_i)_{n<\omega}$. Then for each $i\in I$, $x_i\in \bigcap_{n} f(i)(n)$, so $f(i)\in \mathcal K(X_i)$, thus, $f\in X^*$. Moreover, $f\in \bigcap_n \sigma^*(U_0^*,\dots,U_n^*)$, which is impossible, since the play
\[
\sigma^*(\emptyset),U_0^*,\dots,U_n^*,\sigma^*(U_0^*,\dots,U_n^*),\dots
\]
is compatible with $\sigma^*$.
\end{proof}

As a consequence, we have

\begin{proof}[Proof of Theorem \ref{iprod}]
Since $\mathcal K(X_i)$ is a Baire space with a countable-in-itself $\pi$-base (see the proof of Theorem \ref{krom}), then $\prod_{i\in I} \mathcal K(X_i)$ is a Baire space by \cite[Theorem 5]{Zs0}, and so is $\prod_i X_i$ by our Theorem \ref{krominfi}.
\end{proof}

Recall that $X$ has a {\it base of countable order} (BCO) $\mathcal B$ \cite{WW}, provided each strictly decreasing sequence  of members of $\mathcal B$ containing some $x\in X$ forms a base of neighborhoods at $x$.

\begin{thm}\label{infibetaproduct} Let $I$ be an index set, and for each $i\in I$, $X_i$ be an $R_0$ hereditarily Baire space with a BCO . Then $\prod_i X_i$ is a Baire space.
\end{thm}

\begin{proof} For each $i\in I$, choose a BCO $\mathcal B_i$ for $X_i$ and prove that  {\it $\mathcal K^0_{\mathcal B_i}(X_i)$ is a  dense hereditarily Baire subspace of $\mathcal K(X_i)$}: as for density, take  a decreasing $h\in \mathcal B_i^k$, $k<\omega$, choose $g\in\mathcal K^0_{\mathcal B_i}(X_i)$ with $g(0)\subset h(k)$, and define
\[
f(m)=
\begin{cases}
h(m), \ &\text{if} \ m\le k,\\
g(m-k-1), \  &\text{if} \  m>k.
\end{cases}
\]
Then $f\in [h]\cap \mathcal K_{\mathcal B_i}^0(X_i)$, so $ \mathcal K_{\mathcal B_i}^0(X_i)$ is dense in  $\mathcal K(X_i)$,

To show that  $\mathcal K^0_{\mathcal B_i}(X_i)$ is a  hereditarily Baire space, we will use that, by \cite[Corollary 3.9]{Zs1}, in spaces with a BCO, hereditary Baireness is equivalent to $\beta$-unfavorability in the strong Choquet game: indeed,  assume that $\sigma_i^*$ is a winning strategy for $\beta$ in $Ch(\mathcal K_{\mathcal B_i}^0(X))$, and define a
strategy $\sigma_i$ for $\beta$ in $Ch(X_i)$ as follows: if $\sigma_i^*(\emptyset)=(f_0,V_0^*)$ for some $f_0\in\mathcal K_{\mathcal B_i}^0(X)$ and $V_0^*=[f_0\restriction_{m_0}]\cap \mathcal K_{\mathcal B_i}^0(X)$, where $m_0\ge 1$, then pick $x_0\in \bigcap_nf_0(n)$, choose $V_{0}\in\mathcal B_i$ so that $x_0\in V_{0}\subsetneq f_0(m_0-1)$, if
$f_0(m_0-1)$ is not a singleton, and $V_{0}=f_0(m_0-1)$, if $f_0(m_0-1)$ is a singleton, and define
$\sigma_i(\emptyset)=(x_0,V_0)$.
If $x_0\in U_0\subseteq V_0$ for some $U_0\in \mathcal B_i$,  let $n_0\ge m_0$ be such  that $f_0(n_0)\subseteq U_0$, and consider $\sigma_i^*([f_0\restriction_{n_0+1}]\cap \mathcal K_{\mathcal B_i}^0(X))=(f_1,V_1^*)$, where $f_1\in\mathcal K_{\mathcal B_i}^0(X)$ and $V_1^*=[f_1\restriction_{m_1}]\cap \mathcal K_{\mathcal B_i}^0(X)$ for some $m_1\ge n_0+1$. Pick $x_1\in\bigcap_n f_1(n)$, choose $V_{1}\in\mathcal B_i$ so that $x_1\in V_{1}\subsetneq f_1(m_1-1)$, if$f_1(m_1-1)$ is not a singleton, and $V_{1}=f_1(m_1-1)$, if $f_1(m_1-1)$ is a singleton, and define $\sigma_i(U_0)=(x_1,V_1)$.
Proceeding inductively, assume that for a given $k\ge 1$ and all $j< k$, $\sigma_i(U_0,\dots,U_j)=(x_{j+1},V_{j+1})$ has been defined, along with $V_{j+1}^*=[f_{j+1}\restriction_{m_{j+1}}]\cap \mathcal K_{\mathcal B_i}^0(X)$ and $m_{j+1}>n_j\ge m_j$ so that $f_j(n_j)\subseteq U_j$
and  $\sigma_i^*([f_0\restriction_{n_0+1}]\cap \mathcal K_{\mathcal B_i}^0(X),\dots,[f_j\restriction_{n_j+1}]\cap \mathcal K_{\mathcal B_i}^0(X))=(f_{j+1},V_{j+1}^*)$, where   $V_j$ is either a singleton or a proper subset of $f_j(m_j-1)$. Let $U_k\in\mathcal B_i$ be such that $x_k\in U_k\subseteq V_k$, and find $n_k\ge m_k$ such that $f_k(n_k)\subseteq U_k$
Consider
\[
\sigma_i^*([f_0\restriction_{n_0+1}]\cap \mathcal K_{\mathcal B_i}^0(X),\dots,[f_k\restriction_{n_k+1}]\cap \mathcal K_{\mathcal B_i}^0(X))=(f_{k+1},V_{k+1}^*),
\]
where $f_{k+1}\in\mathcal K_{\mathcal B_i}^0(X)$ and $V_{k+1}^*=[f_{k+1}\restriction_{m_{k+1}}]\cap \mathcal K_{\mathcal B_i}^0(X)$ for some $m_{k+1}\ge n_k+1$.
Pick $x_{k+1}\in\bigcap_n f_{k+1}(n)$, choose $V_{k+1}\in\mathcal B_i$ so that $x_{k+1}\in V_{k+1}\subsetneq f_{k+1}(m_{k+1}-1)$, if
$f_{k+1}(m_{k+1}-1)$ is not a singleton, and $V_{k+1}=f_{k+1}(m_{k+1}-1)$, if $f_{k+1}(m_{k+1}-1)$ is a singleton,
and put $\sigma_i (U_0,\dots,U_k)=(x_{k+1},V_{k+1})$. We will show that $\sigma_i$ is a winning strategy for $\beta$ in $Ch(X_i)$: indeed, let
\[
(x_0,V_0),U_0,\dots,(x_k,V_k),U_k,\dots
\]
be a play of $Ch(X_i)$ compatible with $\sigma_i$, and assume  $\bigcap_k V_k\neq\emptyset$.  Define $f\in\downarrow\mathcal B_i^{\omega}$ as follows: for all $k<\omega$ and $m_{k-1}\le p< m_k$ put
$f(p)=f_k(p)$ (for completeness, let $m_{-1}=0$). Then $\bigcap_p f(p)=\bigcap_k f(m_k-1)=\bigcap_k f_k(m_k-1)\supseteq \bigcap_kV_k$, so $f\in \mathcal K_{\mathcal B_i}^0(X)$,  since by the construction of $\sigma_i$, $\{V_n:n<\omega\}$ is either a strictly decreasing sequence of elements of $\mathcal B_i$, or  a singleton.
Moreover, $f\restriction_{m_k}=f_k\restriction_{m_k}$, thus, $f\in \bigcap_k V_k^*$, which is impossible, since
\[
(f_0,V_0^*),[f_0\restriction_{n_0+1}]\cap \mathcal K_{\mathcal B_i}^0(X),\dots,(f_k,V_k^*),[f_k\restriction_{n_k+1}]\cap \mathcal K_{\mathcal B_i}^0(X),\dots
\]
is a play of $Ch(\mathcal K_{\mathcal B_i}^0(X))$ compatible with $\sigma_i^*$.

It now follows from \cite[Theorem 1.1]{ChP} that $\prod_{i\in I}\mathcal K^0_{\mathcal B_i}(X_i)$ is a Baire space, which is also dense in
$\prod_{i\in I} \mathcal K(X_i)$, which it turn implies that  $\prod_{i\in I} X_i$ is a Baire space by Theorem \ref{krominfi}.
\end{proof}

\end{document}